\documentclass[a4paper]{article}

\usepackage{amsmath, amsthm, amsfonts, amssymb, accents}
\usepackage{graphicx,color}
\usepackage{srcltx}
\usepackage{dsfont}

\theoremstyle{plain}
\newtheorem{theorem}{Theorem}[section]

\newtheorem{lemma}[theorem]{Lemma}

\theoremstyle{remark}

\numberwithin{equation}{section}

\def\tht{\theta}
\def\Om{\Omega}

\def\e{\varepsilon}
\def\g{\gamma}
\def\G{\Gamma}
\def\l{\lambda}
\def\p{\partial}
\def\D{\Delta}

\def\E{\mbox{\rm e}}
\def\a{\alpha}

\def\Odr{\mathcal{O}}
\def\H{W_2}

\def\Ho{W_{2,0}}

\def\di{\,d}

\def\iu{\mathrm{i}}

\def\op#1{\mathcal{#1}}

\def\efop{\op{H}^{\mathrm{eff}}}

\DeclareMathOperator{\RE}{Re}

\DeclareMathOperator{\Dom}{\mathfrak{D}}

\newcounter{assumption}
\begin{document}
\allowdisplaybreaks

\title{\textbf{Planar waveguide with ``twisted'' boundary conditions: small width}}
\author{Denis Borisov\,$^a$, Giuseppe Cardone$^b$}
\date{\small
\begin{center}
\begin{quote}
\begin{enumerate}
{\it
\item[$a)$]
Institute of Mathematics of Ufa Scientific Center of RAS, Chernyshevskogo st.~112, 450008, Ufa, Russia Federation
\\
Bashkir State Pedagogical University,
October St.~3a, 450000 Ufa,
Russian Federation; \texttt{borisovdi@yandex.ru}
\item[$b)$]
University of Sannio,
Department of Engineering, Corso Garibaldi,
107, 82100 Benevento, Italy; \texttt{giuseppe.cardone@unisannio.it}
}
\end{enumerate}
\end{quote}
\end{center}
%
%
}
\maketitle

\begin{abstract}
We consider a planar waveguide with ``twisted'' boundary conditions. By twisting we mean a special combination of Dirichlet and Neumann boundary conditions. Assuming that the width of the waveguide goes to zero, we identify the effective (limiting) operator as the width of the waveguide tends to zero, establish the uniform resolvent convergence in various possible operator norms, and give the estimates for the rates of convergence. We show that studying the resolvent convergence can be treated as a certain threshold effect and we present an elegant technique which justifies such point of view.
\end{abstract}
%

\section{Introduction}

In this paper we study a model of a planar waveguide with twisted boundary conditions. The waveguide is modeled by a strip of a small width. In this domain we consider the Laplacian with a special combination of the Dirichlet and Neumann condition, see fig.~\ref{fig1}. The parameter $L$ introduced on fig.~\ref{fig1} is assumed to be either fixed or defined as $L=\e\ell$ for a fixed $\ell$, where $\e$ is the width of the strip. Our main aim is to study the asymptotic behavior of the resolvent of such operator as the width of the waveguide tends to zero.

There is a vast literature devoted to the study of various elliptic operator in thin bounded domains. Not aiming to cite all existing papers and books, we just mention the books of S.A.~Nazarov and G.P.~Panasenko \cite{N1}, \cite{P1}, see also the references in these books and other papers of these authors. In these works the most attention was paid to the case of Neumann problems and the behavior of the spectrum was studied. Similar studies but for Dirichlet problems were made in \cite{BF1}, \cite{BF2}, \cite{BC}, \cite{BM1}, \cite{BM2}, \cite{FK}, \cite{FS},  \cite{Gr1}, \cite{Gr2}, \cite{Gru1}, \cite{N3}, \cite{N2},  \cite{OV}. The uniform resolvent convergence for Dirichlet Laplacian in a thin bounded two-dimensional domain was established in \cite{FS}, while the multidimensional case was treated in \cite{BF2}. We also mention the paper \cite{EP3}, where the Laplace-Beltrami operator was studied on a bounded manifold shrinking to a finite graph.

The case of unbounded thin domains was considered, too. Here we can refer to the papers \cite{D1}, \cite{ED}, \cite{EP2}, \cite{EP1},  \cite{FS2}, \cite{FS3}, \cite{Gr2}, \cite{Gru2}, \cite{K1}. In all cases the geometry of the thin domain was nontrivial in the sense that it could not be treated just by the separation of variables. The first and on of the most popular examples is a curved infinite thin strip or tube. Such model was studied in \cite{ED} with Dirichlet condition in two- and three-dimensional case. Two-dimensional case with Dirichlet and Neumann conditions on the opposite sides of the strip was considered in \cite{K1}. The main results of the mentioned papers are the estimates for the number of bound states and their asymptotic expansions. The uniform resolvent convergence to the effective operator was also established. A more general case, namely, a curved infinite tube with a torsion was studied in \cite{Gru2}. Here the complete asymptotic expansions for the eigenvalues were constructed. The curved thin tube whose cross-section has a hole was treated in \cite{D1}. The quasi-classical approximation was considered and the operator was multiplied by the square of a small parameter characterizing the width of the tube. The asymptotic expansions for the eigenvalues and the eigenfunctions were constructed. One more result was the asymptotic expansion to the solution of the initial evolution problem.

One more example of an infinite thin domain is a thin domain with variable width. It was studied in \cite{FS2}. The uniform resolvent convergence to an effective operator was established as well as the estimates for the rate of convergence. One more result of \cite{FS2} is two-terms asymptotics for the first eigenvalues. Similar results but in the case of periodically curved thin strip were established in \cite{FS3}. Thin domains obtained as appropriate approximations of various graphs were treated in \cite{EP2}, \cite{EP1}. In \cite{EP1}  the  resolvent convergence and the effective operator were studied, while in \cite{EP2} the study was devoted to the asymptotic behavior of the resonances. Various physical aspects of the elliptic operators in thin domains were discussed in \cite{D2}, \cite{D3}. We also mention the review \cite{Gr2}, where one can find further information of the state-of-art in the studies of thin domains.

The models similar to our were studied in \cite{ACF}, \cite{CE}, \cite{CF}, \cite{Ca}, \cite{C}. The first four papers are devoted to the model of a thin bent waveguide. The curvature describing the bending was scaled together with the width in the same fashion as we rescale our waveguide in the case $L=\e\ell$. In \cite{C} one more similar model was considered. Here the waveguide was three-dimensional and the nontrivial geometry came from localized twisting which was scaled together with the width as the bending in \cite{ACF}, \cite{CE}, \cite{CF}, \cite{Ca}. The operator in \cite{ACF}, \cite{CE}, \cite{Ca}, \cite{C} was the Dirichlet Laplacian, while in \cite{CF} it was the Robin Laplacian. The main result of \cite{ACF}, \cite{CE}, \cite{CF}, \cite{Ca}, \cite{C} is the convergence theorems. Namely, the effective operator was found and the uniform resolvent convergence was proven. In \cite{Ca} the estimates for the rate of convergence were   established under some additional restrictions for the resolvent's domain. It was also shown in \cite{ACF}, \cite{CE}, \cite{CF}, \cite{Ca} that the effective operator can involve nontrivial boundary condition instead of bending, if certain one-dimensional operators possesses eigenvalue or resonance at zero.

The main difference of our model in comparison with \cite{ACF}, \cite{CE}, \cite{CF}, \cite{Ca} is that instead of bending we consider a special ``twisted'' combination of Dirichlet and Neumann boundary conditions. It must be said that we have borrowed the idea of such combination from \cite{DK}. In the case $L=\e\ell$ our model can be also considered as a two-dimensional analogue of that in \cite{C}.

Let us describe our main results. In the case $L=\e\ell$ we can rescale the strip to that with a fixed width and the boundary conditions imposed on fixed parts of the boundary. Under such rescaling the original resolvent $(\op{H}^{(\e)}_{\e\ell}-E\e^{-2}-\l)^{-1}$, $E=\mathrm{const}$, $\l=\mathrm{const}$, becomes $\e^2(\op{H}^{(1)}_{\ell}-E-\e^2\l)^{-1}$. Here $\op{H}^{(\e)}_{L}$ denotes the operator we consider. As $E$
we choose the threshold of the essential spectrum and we consider the original question on the resolvent convergence as a certain threshold effect. Our results show that the asymptotic behavior of the original resolvent highly depends on the spectral properties of the threshold $E$ of the rescaled operator. The idea of treating the resolvent convergence for perturbed elliptic operators as a certain threshold effect has been recently developed in the series of papers by M.Sh.~Birman and T.A.~Suslina for the homogenization problems,  see, for instance, \cite{BS1}, \cite{BS2}, \cite{BS3}. Although we study a problem of a completely different nature and we employ an essentially different technique, we show that in our case the effective operator also appears as a result of certain threshold effect. Namely, the form of the effective operator depend on whether the considered threshold is a virtual level or not. If it is not, the effective boundary condition for the effective operator is the  Dirichlet one. If the virtual level is present, the effective boundary condition becomes more complicated. The last fact is in a good accordance with the results of \cite{ACF}, \cite{CE}, \cite{CF},  \cite{Ca} since our virtual levels play the same role as the aforementioned zero eigenvalues or resonances in \cite{ACF}, \cite{CE}, \cite{CF}, \cite{Ca}. The same situation occurs in \cite{C}. Namely, if we rescale  the tube considered in \cite{C} to a fixed one, we obtain the fixed tube with a fixed twisting. It is known that such model has no virtual levels at the threshold of the essential spectrum and the Hardy inequality is valid, see \cite{EKK}. This is why the effective operator in \cite{C} involves the Dirichlet condition.

If $L$ is fixed, we again make the aforementioned rescaling of the strip, and it leads us to the operator $\op{H}^{(1)}_{L\e^{-1}}$, i.e., for this operator the length of the overlap of the Neumann conditions increases unboundedly as $\e\to+0$. And as in the first case, we reduce the question on determining the effective operator to studying the spectral properties of the thresholds of certain fixed operator.

In addition to identifying the effective operator, we prove the uniform resolvent convergence of the perturbed operator to the effective one. Moreover, we   establish the estimates for the rates of convergence. These results are obtained for two possible operator norms in which we can consider the resolvent convergence. Namely, these are the norms of the operators acting in $L_2$ and from $L_2$ into $\H^1$. We also observe that in \cite{ACF}, \cite{CE}, \cite{Ca}, \cite{C} the uniform resolvent convergence was established only in the sense of $L_2$ norm.

\begin{figure}
\begin{center}\includegraphics[scale=0.6]{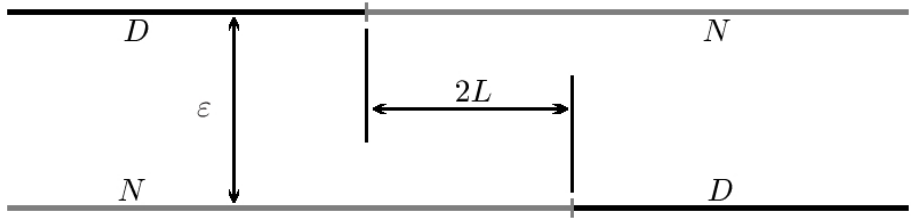}

\caption{Waveguide with combined boundary conditions} \label{fig1}\end{center}
\end{figure}

The approach we use is quite elegant. The core is the technique presented in \cite{pap1} for studying the same model but in the case of the fixed width. It comes originally from the papers \cite{BEG}, \cite{GR2}, where it was used to study the behavior of the discrete eigenvalues emerging from the essential spectrum. In \cite{pap1} it was adapted also for the considered model of the waveguide of a fixed width with twisted boundary conditions. In this paper we apply the adapted technique to study the resolvent convergence and this is for the first time that this approach is used for such study. Its main content is as follows. In the case $L=\e\ell$ it allows us to make an analytic continuation of the resolvent $(\op{H}^{(1)}_{\ell}-E-\e^2\l)^{-1}$ in a vicinity of the threshold of the essential spectrum. And then we give the description of the possible singularities of this continuation at the threshold. Exactly the last description determines the effective boundary condition and the rates of the resolvent convergence. If $L$ is fixed, we again employ the same approach but with the combination of some ideas of \cite{G1}, \cite{G2}. We also mention that one of the effective approaches of studying the asymptotic behavior of the solutions to the problems in the thin domains is the method of matching of asymptotic expansions, see, for instance, \cite{N1}. Nevertheless, this method does not work in our case since it requires a quite high smoothness of the solution to the limiting problem that is not the case for our problem. In conclusion we note that our approach is quite universal and can be employed in studying various similar problems.

\section{Formulation of the problem and the main result}\label{Sec.main}

%
Let $x=(x_1,x_2)$ be the Cartesian coordinates in $\mathds{R}^2$, $\e$ be a small positive parameter, and $\Pi^{(\e)}:=\{x: 0<x_2<\e\}$ be an infinite strip of the width $\e$, where $\e$ is a small positive parameter. Given a number $L>0$, we partition the boundary of $\Pi^{(\e)}$ as follows,
\begin{equation*}
\g^{(\e)}_L:=\{x: x_1>L, x_2=0\}\cup\{x: x_1<-L, x_2=\e\},\quad \G^{(\e)}_L:=\p\Pi^{(\e)}\setminus\overline{\g^{(\e)}_{L}}.
\end{equation*}

In this paper we consider the Laplacian in $\Pi^{(\e)}$ subject to the Dirichlet boundary condition on $\g^{(\e)}_L$ and to the Neumann one on $\G^{(\e)}_L$, cf. fig.~\ref{fig1}. We define this operator as associated with the symmetric lower-semibounded sesquilinear form
$(\nabla u, \nabla v)_{L_2(\Pi^{(\e)})}$ on 
$\Ho^1(\Pi^{(\e)},\g^{(\e)}_L)$,
where the symbol $\Ho^1(\Om,S)$ is the Sobolev space of the functions in $\H^1(\Om)$ vanishing on $S$. We denote the introduced operator as $\op{H}^{(\e)}_L$.

Our main goal is to study the resolvent convergence of $\op{H}^{(\e)}_L$ as $\e\to+0$. We consider two cases. In the first case we let $L=\e \ell$, where $\ell\geqslant 0$ is a fixed number independent of $\e$. In the second case $L>0$  is fixed and independent of $\e$. The structure of the effective (limiting) operator depends strongly on $L$ and to formulate the main results we introduce additional notations.

Consider the operator $\op{H}^{(1)}_\ell$. It was shown in \cite[Th. 2.2]{pap1} that there exists infinitely many critical values  $0<\ell_1<\ell_2<\ldots<\ell_n<\ldots$ such that for $\ell\in(\ell_n,\ell_{n+1}]$ the operator $\op{H}^{(1)}_\ell$ has precisely $n$ isolated eigenvalues. Given a function $f\in L_2(\Pi^{(\e)})$, we introduce two projections,
\begin{align*}
&(\op{P}_\e f)(x_1,\e):=\int\limits_0^\e f(x)\chi_1^{(\e)}(x)\di x_2,\quad (\op{P}_{\e,L} f)(x_1,\e):=\int\limits_0^\e f(x)\chi_{1,L}^{(\e)}(x)\di x_2, 
\\
&
\chi_1^{(\e)}(x):=\left\{
\begin{aligned}
\sqrt{\frac{2}{\e}}& \sin \frac{\pi x_2}{2\e}, && x_1>0,
\\
\sqrt{\frac{2}{\e}}& \cos\frac{\pi x_2}{2\e}, && x_1<0,
\end{aligned}
\right.
\quad\chi_{1,L}^{(\e)}(x):=\left\{
\begin{aligned}
\sqrt{\frac{2}{\e}}& \sin \frac{\pi x_2}{2\e}, && x_1>L,
\\
\frac{1}{\sqrt{\e}}&, && |x_1|< L,
\\
\sqrt{\frac{2}{\e}}& \cos\frac{\pi x_2}{2\e}, && x_1<-L.
\end{aligned}
\right. 
\end{align*}
By $\|\cdot\|_0$ and $\|\cdot\|_1$ we denote the norm of an operator acting respectively in $L_2(\Pi^{(\e)})$ and from $L_2(\Pi^{(\e)})$ in  $\H^1(\Pi^{(\e)})$. The symbol $\Dom$
will be employed to indicate the domain of an operator.

Assume first that $L=\e\ell$ and let us introduce the effective operator. It is the Schr\"odinger operator on the axis
\begin{equation}\label{2.4}
\efop:=-\frac{d^2}{dx_1^2}
\end{equation}
subject to certain boundary condition. The type of this condition depends on $\ell$. If $\ell$ is noncritical, the boundary condition is the Dirichlet one at zero, i.e., in this case the domain of the effective operator is given by the identity
\begin{equation*}
\Dom(\efop)=\{u\in\H^1(\mathds{R}): u(0)=0\}\cap\H^2(\mathds{R}_+)\cap\H^2(\mathds{R}_-).
\end{equation*}
If $\ell=\ell_n$ is critical, we have two subcases. Namely, for odd $n$ there is no boundary conditions at all and $\efop$ is the usual Schr\"odinger operator (\ref{2.4}) having $\H^2(\mathds{R})$ as the domain. For even $n$ the boundary conditions are the most complicated and interesting ones. Namely, in this case the domain consists of the functions in $\H^2(\mathds{R}_+)\cap\H^2(\mathds{R}_-)\cap\H^1(\mathds{R})$ satisfying the boundary conditions
\begin{equation}\label{2.6}
u(+0)=-u(-0),\quad u'(+0)=-u'(-0).
\end{equation}

Now we are in the position to formulate our first main result.

\begin{theorem}\label{th2.1}
Assume $L=\e\ell$ and $\l\in \mathds{C}\setminus \mathds{R}$. Then for sufficiently small $\e$ the resolvent $\left(\op{H}^{(\e)}_{\e\ell}-\frac{\pi^2}{4\e^2}-\l\right)^{-1}$ is well-defined and
\begin{equation*}
\left\| \left(\op{H}^{(\e)}_{\e\ell}-\frac{\pi^2}{4\e^2}-\l\right)^{-1} -\chi_1^{(\e)} \big(\efop-\l\big)^{-1}\op{P}_\e
\right\|_0\leqslant C\e^{1/2}
\end{equation*}
for critical $\ell$, and
\begin{align*}
&\left\| \left(\op{H}^{(\e)}_{\e\ell}-\frac{\pi^2}{4\e^2}-\l\right)^{-1} -\chi_1^{(\e)} \big(\efop-\l\big)^{-1}\op{P}_\e
\right\|_0\leqslant C\e^{3/2}
\\
&
\left\| \left(\op{H}^{(\e)}_{\e\ell}-\frac{\pi^2}{4\e^2}-\l\right)^{-1} -\chi_1^{(\e)} \big(\efop-\l\big)^{-1}\op{P}_\e
\right\|_1\leqslant C\e^{1/2}
\end{align*}
for noncritical $\ell$. Here the constants $C$ are independent of $\e$ but depend on $\l$ and $\ell$.
\end{theorem}

Suppose now that $L$ is fixed and independent of $\e$. Here we study the convergence of the resolvent at the point $E\e^{-2}+\l$,  $\l\in \mathds{C}\setminus \mathds{R}$, where $E=0$ or $E=\pi^2/4$. We choose $E$ in this way since $0$ and $\pi^2/(4\e^2)$ are exactly the eigenvalues associated with the first transversal modes on the cross-sections of the waveguides for $|x_1|<L$ and $|x_1|>L$.

It turns out that in the considered case the approximating operator is of different nature than in Theorem~\ref{th2.1} an it also depends substantially on $E$. Namely, by $\efop$ we denote the operator (\ref{2.4}) subject to the Dirichlet condition at $x_1=\pm L$, i.e., it has the domain
\begin{align*} 
\Dom(\efop)=&\{u\in\H^1(\mathds{R}): u(-L)=0,\ u(L)=0\}
\\
&\cap\H^2(-\infty,-L)\cap\H^2(-L,L)\cap\H^2(L,+\infty).
\end{align*}
By $\op{E}_0$ we denote the characteristic function of the segment $[-L,L]$, while $\op{E}_{\pi^2/4}$ is the characteristic functions of $\mathds{R}\setminus[-L,L]$.

\begin{theorem}\label{th2.2}
Assume $L$ is fixed, $\l\in \mathds{C}\setminus \mathds{R}$, and $E=0$ or $E=\pi^2/4$. Then for sufficiently small $\e$ the resolvent $\big(\op{H}^{(\e)}_{L}-E\e^{-2}-\l\big)^{-1}$  is well-defined and satisfies the inequalities
\begin{align*}
&\left\| \left(\op{H}^{(\e)}_{L}-\frac{E}{\e^2}-\l\right)^{-1} -\chi_1^{(\e)}\big(\efop-\l\big)^{-1}\op{E}_E\op{P}_\e
\right\|_0\leqslant C\e^{3/2},
\\
&
\left\| \left(\op{H}^{(\e)}_{L}-\frac{E}{\e^2}-\l\right)^{-1} -\chi_1^{(\e)}\big(\efop-\l\big)^{-1}\op{E}_E\op{P}_\e
\right\|_1\leqslant C\e^{1/2}.
\end{align*}
Here the constants $C$ are independent of $\e$ but depend on $\l$ and $\ell$.
\end{theorem}

Let us discuss the main results. The fact that the action of the effective operator is determined by the identity (\ref{2.4}) is quite expectable and the main nontriviality is in the boundary condition. In the first case the assumption $L=\e\ell$ means that the domain $\Pi^{(\e)}$ with partition of the boundary into the subsets $\g^{(\e)}_{\e\ell}$, $\G^{(\e)}_{\e\ell}$ can be rescaled to the fixed domain $\Pi^{(1)}$ with the fixed partition $\g^{(1)}_{\ell}$, $\G^{(1)}_{\ell}$. And since after rescaling we in fact study the convergence of the operator $\big(\op{H}^{(1)}-\frac{\pi^2}{4}-\e^2\l\big)^{-1}$, the effective boundary condition at zero is determined by the operator $\op{H}^{(1)}_\ell$ in the rescaled domain $\Pi^{(1)}$. More precisely, the effective boundary condition depends on the spectral structure of the threshold $\pi^2/4$ of the essential spectrum of $\op{H}^{(1)}_\ell$. If $\ell$ is critical, it was shown in \cite{pap1} that there exists a virtual level at the threshold of the essential spectrum, see problem (\ref{3.1}) below. And the structure of this solution determines the effective boundary condition in this case. If $\ell$ is non-critical,  virtual levels are absent,  and exactly this fact implies the effective Dirichlet boundary condition at zero.

Similar situation occurs, if $L$ is fixed. In this case we again rescale the domain $\Pi^{(\e)}$ to $\Pi^{(1)}$. The sets $\g^{(\e)}_{L}$ and $\G^{(\e)}_{L}$ become
$\g^{(1)}_{L\e^{-1}}$ and $\G^{(1)}_{L\e^{-1}}$, while the resolvent becomes
$(\op{H}^{(1)}_{L\e^{-1}}-E-\e^2\l)^{-1}$. As $\e$ goes to zero, we employ the ideas of \cite{G1}, \cite{G2} this resolvent behaves approximately as a direct sum of two Laplacians in $\Pi^{(1)}$ subject to the Neumann condition on $\{x: x_2=1\}\cup\{x: x_1<0, x_2=0\}$ and to the Dirichlet condition on $\{x: x_1>0, x_2=0\}$. This operator again has no virtual levels associated with both the spectral points $E=0$ and $E=\pi^2/4$, see Lemmas~\ref{lm3.2},~\ref{lm3.3},~\ref{lm3.4}. And again this fact finally yields the Dirichlet condition at $x_1=\pm L$.

We also
observe that in the case $L$ being fixed, there is one more factor $\op{E}_E$ in the approximation for the original resolvent. And due to the presence of this function for each $f\in L_2(\Pi^{(\e)})$ we have
\begin{align*}
&\big((\efop-\l)^{-1}\op{E}_0\op{P}_\e f\big)(x_1)=0, \quad\ \  |x_1|>L,
\\
&\big((\efop-\l)^{-1}\op{E}_\frac{\pi^2}{4}\op{P}_\e f\big)(x_1)=0, \quad |x_1|<L.
\end{align*}
It means that the action of the effective resolvent is nontrivial only as $|x_1|<L$ for $E=0$ and as $|x_1|>L$ for $E=\pi^2/4$.

As we see, in both cases $L=\e\ell$ and $L=\mathrm{const}$ the effective Dirichlet condition appears once there is no virtual levels for the rescaled operator at the point $E$, where $E$ is involved in the perturbed resolvent. The presence of the virtual level gives rise either to nontrivial boundary conditions (\ref{2.6}) or to the absence of the boundary conditions. We believe that such influence of virtual levels on the effective boundary conditions is a general fact occurring not only in our model. In a more complicated model the threshold can be also an (embedded) eigenvalue. We conjecture that in this case the projection to the associated eigenfunction will be the leading term in the asymptotic expansion for the resolvent. And we also conjecture that the leading term in the asymptotic expansion of the perturbed resolvent will be a pole and the mentioned projection will be the associated residue.

\section{Preliminaries}

In this section we collect a series of auxiliary results which will be employed in the proofs of Theorems~\ref{th2.1},~\ref{th2.2}.

Let $\Pi^{(\e)}_a:=\Pi^{(\e)}\cap\{x: |x_1|<a\}$, $\Pi^{\pm}_a:=\Pi^{(1)}\cap\{x: \pm x_1>\pm a\}$. In what follows as $\e=1$ we omit the superscript $^{(1)}$ in the notations related to the identity $\e=1$. For instance, $\Pi=\Pi^{(1)}$, $\Pi_a=\Pi^{(1)}_a$.

In \cite[Th. 2.3]{pap1} a criterium for $\ell$ to be critical for the operator $\op{H}_\ell$ was proven. Namely, the number $\ell=\ell_n$ is critical, if and only if the boundary value problem
\begin{equation}\label{3.1}
-\Delta \phi_n=\frac{\pi^2\!\!}{4}\,\phi_n \quad\text{in} \quad \Pi,\quad
\phi_n=0\quad \text{on} \quad \g_{\ell_n},\qquad \frac{\p\phi_n}{\p x_2}=0\quad \text{on}\quad \G_{\ell_n}
\end{equation}
has a bounded solution belonging to $\H^1(\Pi_a)$ for each $a>0$ and satisfying the asymptotics
\begin{equation}\label{3.2}
\phi_n(x)=\sin\frac{\pi x_2}{2}+\Odr\big(\E^{-\sqrt{2}\pi x_1}\big),\quad x_1\to+\infty.
\end{equation}
This solution is unique. For even $n$ it is odd w.r.t.
the symmetry transformation
\begin{equation}\label{3.3}
(x_1,x_2)\mapsto (-x_1,1-x_2),
\end{equation}
and is even for odd $n$. The first auxiliary lemma describes certain properties of the functions $\phi_n$.

\begin{lemma}\label{lm3.1}
The identities
\begin{align*}
&\int\limits_0^1 \phi_n(0,x_2) \sin \frac{\pi}{2}x_2 \di x_2=(-1)^{n-1}\int\limits_0^1 \phi_n(0,x_2)\cos\frac{\pi}{2}x_2\di x_2,
\\
&\int\limits_0^1 \frac{\p\phi_n}{\p x_1}(0,x_2) \sin \frac{\pi}{2}x_2 \di x_2=(-1)^{n} \int\limits_0^1 \frac{\p\phi_n}{\p x_1}(0,x_2)\cos\frac{\pi}{2}x_2\di x_2,
\\
&\int\limits_0^1 \frac{\p\phi_n}{\p x_1}(0,x_2) \sin \frac{\pi}{2}x_2 \di x_2-\frac{\pi}{2}\int\limits_0^{\ell_n} \phi_n(x_1,0)\di x_1=0,
\\
&\int\limits_0^1 \frac{\p\phi_n}{\p x_1}(0,x_2)\cos\frac{\pi}{2}x_2\di x_2+ \frac{\pi}{2}\int\limits_{-{\ell_n}}^0 \phi_n(x_1,1)\di x_1=0,
\\
&\int\limits_0^1 \phi_n(0,x_2)  \sin\frac{\pi}{2}x_2 \di x_2 +
\frac{\pi}{2}\int\limits_0^{\ell_n} x_1\phi_n(x_1,0)\di x_1=\frac{1}{2},
\\
&\int\limits_0^1 \phi_n(0,x_2)  \cos\frac{\pi}{2}x_2 \di x_2 -\frac{\pi}{2}\int\limits_{-{\ell_n}}^0 x_1\phi_n(x_1,1)\di x_1=\frac{(-1)^{n-1}}{2}
\end{align*}
hold true.
\end{lemma}

\begin{proof}

The first two identities follow easily from the parity of $\phi_n$ under the symmetry transformation (\ref{3.3}). The others can be obtained by integrating by parts in the integrals
\begin{align*}
&\int\limits_{\Pi_a^+} \sin\frac{\pi}{2} x_2 \left(\D+\frac{\pi^2\!\!}{4}\right) \phi_n \di x =0, && \int\limits_{\Pi_a^-} \cos\frac{\pi}{2} x_2 \left(\D+\frac{\pi^2\!\!}{4}\right) \phi_n \di x=0,
\\
&\int\limits_{\Pi_a^+} x_1\sin\frac{\pi}{2} x_2 \left(\D+\frac{\pi^2\!\!}{4}\right) \phi_n \di x =0, &&\int\limits_{\Pi_a^-} x_1\cos\frac{\pi}{2} x_2  \left(\D+\frac{\pi^2\!\!}{4}\right) \phi_n \di x=0.
\end{align*}
One should also take into consideration (\ref{3.2}) and the parity of $\phi_n$ under (\ref{3.3})  and pass then to the limit as $a\to+\infty$.
\end{proof}

The rest of this section is devoted to the study of two auxiliary problems which will be employed in the proof of Theorems~\ref{th2.1},~\ref{th2.2}. Let $h=h(x)$ be a compactly supported function belonging to $L_2(\Pi)$, $\mu$ be a small complex parameter. Denote $\G_+:=\{x: x_2=1\}$, $\G_-:=\{x: x_1<0,\ x_2=0\}$, $\g_*:=\{x: x_1>0, \ x_2=0\}$. Consider the boundary value problem
\begin{equation}\label{3.5}
(-\D-E+\mu^2) v=h\quad \text{in}\quad \Pi, \quad v=0\quad\text{on}\quad \g_*,\quad \frac{\p v}{\p x_2}=0\quad\text{on}\quad \G_+\cup\G_-,
\end{equation}
with $E=0$ or $E=\pi^2/4$. We assume that the solution behaves at infinity as
\begin{equation}\label{3.6}
\begin{aligned}
& v(x)=c_+(\mu) \E^{-\sqrt{\frac{\pi^2}{4}+\mu^2}x_1} \sin \frac{\pi x_2}{2} + \Odr\left( \E^{-\sqrt{\frac{9\pi^2}{4}+\mu^2}x_1}\right), && x_1\to+\infty,
\\
&v(x)=c_-(\mu)\E^{\mu x_1} + \Odr\left(\E^{\sqrt{\pi^2+\mu^2} x_1}\right), &&  x_1\to-\infty,
\end{aligned}
\end{equation}
for $E=0$, and
\begin{equation}\label{3.9}
\begin{aligned}
& v(x)=c_+(\mu) \E^{-\mu x_1} \sin \frac{\pi x_2}{2} + \Odr\left( \E^{-\sqrt{2\pi^2+\mu^2}x_1}\right), && x_1\to+\infty,
\\
&v(x)=c_-(\mu)\E^{\iu \frac{\pi}{2} x_1} + \Odr\left(\E^{\sqrt{\frac{3\pi^2}{4}+\mu^2} x_1}\right), &&  x_1\to-\infty,
\end{aligned}
\end{equation}
for $E=\pi^2/4$. In both case $c_\pm(\mu)$ are some constants.

In the case $E=0$ the solvability of the similar problem but with the Dirichlet condition on $\G_+$ was studied in \cite[Sec. 5]{Bor-MSb06} (if one assumes in \cite{Bor-MSb06} that $d=\pi$ and the right hand side in \cite[eq. (5.1)]{Bor-MSb06} is even w.r.t. $x_2$). The technique used in this paper is the same as that in \cite[Sec. 4]{pap1}. The type of the boundary condition on $\G_+$ is completely inessential for this technique. This is why all the results established in \cite[Sec. 5]{Bor-MSb06} are valid in our case up to minor changes related to the other boundary condition on $\G_+$. We formulate the needed results below without adducing the proofs.

Let $a>0$ be such that the support of $h$ lies in $\Pi_a$.

\begin{lemma}\label{lm3.2}
Let $E=0$, $\mu$ be complex and sufficiently small, $b>0$ be a fixed number. 
Then the problem (\ref{3.5}), (\ref{3.6}) is uniquely solvable. The operator mapping the function $h$ into the solution of the problem (\ref{3.5}), (\ref{3.6}) is bounded as that from $L_2(\Pi_a)$ into $\H^1(\Pi_b)$
and is holomorphic in $\mu$. As $|x_1|\geqslant a$ the solution (\ref{3.5}), (\ref{3.6}) can be represented as
\begin{equation}\label{3.7}
\begin{aligned}
& v(x)=\sum\limits_{m=1}^{\infty} c_m^+(\mu) \E^{-\sqrt{\pi^2
\left(m-\frac{1}{2}\right)^2+\mu^2} x_1} \sin\pi\left(m-\frac{1}{2}\right) x_2, && x_1\geqslant a,
\\
&v(x)=c_0^-(\mu) \E^{\mu x_1} + \sum\limits_{m=1}^{\infty} c_m^-(\mu) \E^{\sqrt{\pi^2 m^2+\mu^2} x_1} \sin\pi\left(m-\frac{1}{2}\right)(1-x_2), && x_1\leqslant -a,
\end{aligned}
\end{equation}
where the coefficients $c_m^\pm$ are holomorphic in $\mu$ and satisfy the uniform in $\mu$ estimate
\begin{equation}\label{3.8}
|c_0^-(\mu)|^2+ \sum\limits_{m=1}^{\infty} m \big(|c_m^+(\mu)|^2+|c_m^-(\mu)|^2\big) \leqslant C\|h\|_{L_2(\Pi_a)}^2.
\end{equation}
\end{lemma}

In the case $E=\pi^2/4$  one can again employ the same technique from \cite[Sec. 5]{Bor-MSb06}. All the calculations remain true the same up to some minor changes. The only substantial change is the proof of an analogue of Lemma~5.3 from \cite{Bor-MSb06}. Although the idea of the proof is the same, the changes are not so minor. This is why we prove below the needed statement.

\begin{lemma}\label{lm3.3}
The problem (\ref{3.5}), (\ref{3.9}) with $E=\frac{\pi^2}{4}$, $\mu=0$ and $h=0$ has the trivial solution only.
\end{lemma}

\begin{proof}
Suppose such solution exists. In the same way as in \cite[Lm. 4.2]{Bor-MSb06} one can check that in a vicinity of zero this solution behaves as
\begin{equation*}
v(x)=\a r^{1/2}\sin  \frac{\tht}{2}  + \Odr(r),\quad r\to 0,
\end{equation*}
where $(r,\tht)$ are the polar coordinates associated with $x$. Bearing this fact and (\ref{3.5}), (\ref{3.9}) in mind, we take any fixed $a>0$ and integrate by parts as follows,
\begin{equation}\label{3.11}
\begin{aligned}
0=&\int\limits_{\Pi_a}  x_1 v \left(\D-\frac{\pi^2}{4}\right) \frac{\p \overline{v}}{\p x_1} \di x
\\
=&\int\limits_0^1 \left(x_1 v \frac{\p^2 \overline{v}}{\p x_1^2}- \frac{\p \overline{v}}{\p x_1} \frac{\p}{\p x_1} (x_1 v) \right)\Bigg|_{x_1=-a}^{x_1=a}\di x_2 + 2 \int\limits_{\Pi_a} \left|\frac{\p v}{\p x_1}\right|^2\di x.
\end{aligned}
\end{equation}
By the separation of variables for $x_1<0$ we can represent $v$ as
\begin{equation*}
v(x)=c_-(0)\E^{\iu \frac{\pi}{2} x_1} + v^\bot(x),
\end{equation*}
where
\begin{equation*}
\int\limits_0^1 v^\bot(x)\di x_2=0\quad\text{for each}\quad x_1\in(-\infty,0).
\end{equation*}
We substitute this representation into (\ref{3.11}) and pass to the limit as $a\to+\infty$. It yields
\begin{align*}
0=&-\frac{\iu \pi |c_-(0)|^2}{2} + 2\int\limits_{\Pi\cap\{x: x_1<0\}} \left|\frac{\p v}
{\p x_1}\right|^2 \di x + 2 \int\limits_{\Pi\cap\{x: x_1<0\}} \left|\frac{\p v^\bot}{\p x_1}\right|\di x.
\end{align*}
Hence,
\begin{equation*}
c_-(0)=0,\quad \frac{\p v^\bot}{\p x_1}=0 \quad \text{as}\quad x_1<0,\quad \frac{\p v}{\p x_1}=0\quad \text{as} \quad x_1>0.
\end{equation*}
It implies that the function $v$ is independent of $x_1$ and thus $v\equiv0$.
\end{proof}
All other arguments of \cite[Sec. 5]{Bor-MSb06} can be easily adapted to the problem (\ref{3.5}), (\ref{3.6}) with $E=\frac{\pi^2}{4}$. The result is formulated in the next

\begin{lemma}\label{lm3.4}
Let $E=\pi^2/4$, $\mu$ be complex and sufficiently small, $b>0$ be a fixed number. 
Then the problem (\ref{3.5}), (\ref{3.9}) is uniquely solvable. The operator mapping the function $h$ into the solution of the problem (\ref{3.5}), (\ref{3.9}) is bounded as that from $L_2(\Pi_a)$ into $\H^1(\Pi_b)$ 
and is holomorphic in $\mu$. As $|x_1|\geqslant a$ the solution (\ref{3.5}), (\ref{3.9}) can be represented as
\begin{equation}\label{3.14}
\begin{aligned}
& v(x)=c_0^+(\mu) \E^{-\mu x_1} \sin \frac{\pi x_2}{2}
\\
&\hphantom{v(x)=}+  \sum\limits_{m=2}^{\infty} c_m^+(\mu) \E^{-\sqrt{\pi^2
\left(m-\frac{1}{2}\right)^2-\frac{\pi^2}{4}+\mu^2} x_1} \sin\pi\left(m-\frac{1}{2}\right) x_2, && x_1\geqslant a,
\\
&v(x)= c_0^-(\mu)\E^{\iu \sqrt{\frac{\pi^2}{4}-\mu^2} x_1} 
\\
&\hphantom{v(x)=}+ \sum\limits_{m=2}^{\infty} c_m^-(\mu) \E^{\sqrt{\pi^2 m^2-\frac{\pi^2}{4}+\mu^2} x_1} \sin\pi\left(m-\frac{1}{2}\right)(1-x_2), && x_1\leqslant -a,
\end{aligned}
\end{equation}
where the coefficients $c_m^\pm$ are holomorphic in $\mu$ and satisfy the uniform in $\mu$ estimate (\ref{3.8}).
\end{lemma}

\section{Proof of Theorem~\ref{th2.1}}

In this section we prove Theorem~\ref{th2.1}. Given $f\in L_2(\Pi^{(\e)})$, we denote
\begin{equation*}
u_\e:=\left(\op{H}^{(\e)}_{\e\ell}-\frac{\pi^2}{4\e^2}-\l\right)^{-1} f.
\end{equation*}
We rescale the variables $x\mapsto x\e^{-1}$ keeping the notation ``$x$'' for the new rescaled variables. It leads us to another representation for $u_\e$,
\begin{equation}\label{4.1}
u_\e=\e^2 \widetilde{u}_\e(\e\cdot),\quad \widetilde{u}_\e:= \left(\op{H}_{\ell}-\frac{\pi^2}{4}-\e^2\l\right)^{-1} f(\e\cdot).
\end{equation}
In what follows we employ exactly this representation. 

At our next step we introduce the function $\widetilde{u}_1$ as the solution to the boundary value problem
\begin{equation}\label{v3.3}
\begin{gathered}
\left(-\D-\frac{\pi^2}{4}-\e^2\l\right)\widetilde{u}_1=f\quad\text{in}\quad \Pi,
\\
\widetilde{u}_1=0\quad\text{on}\quad\g_0\cup\{x: x_1=0,\ 0<x_2<1\},\qquad
\frac{\p \widetilde{u}_1}{\p x_2}=0\quad\text{on}\quad \G_0.
\end{gathered}
\end{equation}
The function $u_1$ can be constructed explicitly by the separation of variables,
\begin{align}
&\widetilde{u}_1(x,\e)=\sum\limits_{m=1}^{\infty} U_m(x,\e)\chi_m(x), \label{v3.4}
\\
&U_m(x_1,\e)=\left\{
\begin{aligned}
&  \int\limits_0^{+\infty} \frac{\E^{-k_m|x_1-t_1|}-\E^{-k_m(x_1+t_1)}}
{2k_m} f_m(\e t_1)\di t_1, && x_1>0,
\\
&  \int\limits_{-\infty}^0 \frac{\E^{-k_m|x_1-t_1|}-\E^{k_m(\e\sqrt{-\l}) (x_1+t_1)}}{2 k_m} f_m(\e t_1)\di t_1, && x_1<0,
\end{aligned}
\right.\label{v4.43}
\\
&E_m:=\pi^2\left(m-\frac{1}{2}\right)^2,\quad
\chi_m(x):=\left\{
\begin{aligned}
&\sqrt{2}\sin\pi\left(m-\frac{1}{2}\right) x_2, && x_1>0,
\\
&\sqrt{2}\sin\pi\left(m-\frac{1}{2}\right)(1-x_2), && x_1<0,
\end{aligned}
\right. \nonumber
\\
&k_0:=\e\sqrt{-\l},\quad k_m:=\sqrt{E_m-E_1-\e^2\l},
\quad
f_m(\e x_1):=\int\limits_0^1 \chi_m(x) f(\e x)\di x_2. \nonumber
\end{align}
where the branch of the square root in the definition of $k_m$ is fixed by the requirement $\sqrt{1}=1$.

The series (\ref{v3.4}) converges in $\H^2(\Pi\setminus\{x:  x_1=0\})\cap\H^1(\Pi)$ that can be shown by analogy with \cite[Lm 3.1]{Bor-MSb06}. It also clear that the Parseval identity
\begin{equation}
\|f\|_{L_2(\Pi^{(\e)})}^2=\e^2\sum\limits_{m=1}^{\infty} \|f_m(\e\,\cdot)\|_{L_2(\mathds{R})}^2
\label{v3.25}
\end{equation}
holds true. Hereinafter by $C$ we indicate various inessential constants independent of $\e$ and $f$.

\begin{lemma}\label{vlm3.8}
Let $a>0$ be a fixed number. The estimates
\begin{align}
&\|\widetilde{u}_1-U_1\chi_1\|_{\H^1(\Pi)}\leqslant C 
\e^{-1}\|f\|_{L_2(\Pi^{(\e)})},\label{v3.24a}
\\
&\|U_1\chi_1\|_{\H^1(\Pi_a)}\leqslant C 
\e^{-3/2} \|f\|_{L_2(\Pi_\e)}\label{v3.24b}
\end{align}
hold true. 
\end{lemma}

\begin{proof}
The functions $U_m$ solve the problems
\begin{equation*}
-U''_m+(E_m-E_1-\e^2\l)U_m=f_m(\e\,\cdot),\quad x\in\mathds{R}\setminus\{0\},\quad U_m(0)=0,
\end{equation*}
and therefore satisfy the equations
\begin{equation*}
\|U'_m\|_{L_2(\mathds{R}_\pm)}^2 +(E_m-E_1-\e^2\l)\|U_m\|_{L_2(\mathds{R}_\pm)} ^2 = (f_m^{(\pm)}(\e\,\cdot),U_m)_{L_2(\mathds{R}_\pm)}.
\end{equation*}
In its turn, they imply
\begin{equation}\label{v4.1}
\|U_m\|_{L_2(\mathds{R}_\pm)}\leqslant \frac{\|f_m(\e\cdot)\|_{L_2(\mathds{R}_\pm)}}{E_m-E_1-\e\RE\l},\quad
\|U_m'\|_{L_2(\mathds{R}_\pm)}\leqslant \frac{\|f_m(\e\cdot)\|_{L_2(\mathds{R}_\pm)}}{\sqrt{E_m-E_1-\e\RE\l}}.
\end{equation}
These estimates and   (\ref{v3.25}) yield (\ref{v3.24a}).

To prove the second estimate, we represent $U_1$ as follows,
\begin{equation}
\begin{aligned}
&U_1(x,\e)=\E^{-\mu y_1} \int\limits_0^{x_1} \frac{\sinh \mu t_1}{\mu} f_1(\e t_1)\di t_1+\frac{\sinh\mu x_1}{\mu} \int\limits_{x_1}^{+\infty} \E^{-\mu t_1} f_1(\e t_1)\di t_1,\quad x_1>0
\\
&U_1(y,\e)=-\frac{\sinh\mu y_1}{\mu} \int\limits_{-\infty}^{x_1} \E^{\mu t_1} f_1(\e t_1)\di t_1+\E^{\mu x_1} \int\limits_{x_1}^0 \frac{\sinh \mu t_1}{\mu} f_1(\e t_1)\di t_1,\quad x_1<0,
\end{aligned}\label{v4.5}
\end{equation}
where $\mu=\e\sqrt{-\l}$. Employing this representation, the Parseval identity (\ref{v3.25}), and the Schwarz inequality, one can prove easily the estimate (\ref{v3.24b}).
\end{proof}

We construct $\widetilde{u}_\e$ as
\begin{equation}\label{v3.5a}
\widetilde{u}_\e=\widetilde{u}_1\xi_1+\widetilde{u}_2,
\end{equation}
where $\xi_1=\xi_1(x_1)$ is an infinitely differentiable cut-off function being one as $|x_1|>\ell+2$ and vanishing as $|x_1|<\ell+1$.  In  view of (\ref{v3.3}) the function $\widetilde{u}_2$ is given by the formulas
\begin{equation*}
\widetilde{u}_2=(\op{H}_\ell-E_1-\e^2\l)^{-1}g, \quad g:=f-(-\D-E_1-\e^2\l)\widetilde{u}_1\xi \in L_2(\Pi).
\end{equation*}
It follows from Lemma~\ref{vlm3.8} that
\begin{equation}\label{v3.18}
\|\widetilde{u}_1\|_{\H^1(\Pi_{\ell+2})}\leqslant  C\e^{-3/2}\|f\|_{L_2(\Pi^{(\e)})}.
\end{equation}
The last inequality  and the relation
\begin{equation*}
g=(1-\xi_1)f+\left(2\xi_1'\frac{\p}{\p x_1}+\xi_1''\right)\widetilde{u}_1
\end{equation*}
yield that the support of the
function $g$ is contained in $\Pi_{\ell+2}$ and the inequality
\begin{equation}\label{v4.31}
\|g\|_{L_2(\Pi_{\ell+2})}\leqslant C \e^{-3/2} \|f\|_{L_2(\Pi^{(\e)})}
\end{equation}
holds true.

Now we employ the results of \cite{pap1} and this is the crucial point in the proof. Namely, we apply Lemma~4.7 from the cited paper to the function $u_2$. It implies the following representation,
\begin{gather}
\widetilde{u}_2=\frac{\op{T} g}{\e\sqrt{-\l}} \phi + \widehat{u}_2,\label{v3.15}
\\
\|\widehat{u}_2\|_{L_2(\Pi_b)}\leqslant C
\|g\|_{L_2(\Pi_{\ell+2})}
\leqslant C 
\e^{-3/2} \|f\|_{L_2(\Pi^{(\e)})},
\label{v3.23}
\end{gather}
where 
the number $b>0$ is arbitrary but fixed, and in the estimate for $\widehat{u}_2$ we used (\ref{v4.31}). The functional $\op{T}$ is defined as
\begin{equation}\label{4.2}
\begin{aligned}
&\op{T}g:=\int\limits_{\Pi} g \phi_n\di x, && \phi=\phi_n, && \text{if}\quad \ell=\ell_n\quad\text{is critical},
\\
&\op{T}g:=0, && \phi=0, && \text{if}\quad \ell\quad\text{is not critical}.
\end{aligned}
\end{equation}
We remind that for critical $\ell$ the function $\phi_n$ is defined as the unique solution to the problem (\ref{3.1}).

Denote
\begin{equation*}
F_\pm(\e):= \frac{\e^{3/2}}{2\sqrt{-\l}}\int\limits_{\mathds{R}_\pm}
\E^{-\e\sqrt{-\l}|t_1|} f_1(\e t_1)\di t_1,\quad \mathds{R}_\pm:=\{t: \pm t\geqslant 0\}.
\end{equation*}

Since
\begin{equation*}
f_1(\e t_1)=\e^{-1/2}(\op{P}_\e f)(t_1,\e),
\end{equation*}
we see that
\begin{equation}\label{v4.34}
F_\pm(\e)= \frac{1}{2\sqrt{-\l}}\int\limits_{\mathds{R}_\pm}
\E^{-\sqrt{-\l}|t_1|} (\op{P}_\e f_1)(t_1,\e)\di t_1,
\end{equation}
Thus,
\begin{equation}
|F_+(\e)|+|F_-(\e)|\leqslant C\|f\|_{L_2(\Pi^{(\e)})}. \label{v4.35}
\end{equation}

\begin{lemma}\label{vlm4.4}
For critical $\ell=\ell_n$ the estimate
\begin{equation*}
\big|\op{T} g-\e^{-3/2}\sqrt{-2\l}\big(F_+ + (-1)^{n-1} F_-\big)\big|\leqslant C\e^{-1}\|f\|_{L_2(\Pi^{(\e)})}
\end{equation*}
holds true.
\end{lemma}

\begin{proof}
 We represent $g$ as
\begin{equation*}
g=(-\D-E_1-\e^2\l) (1-\xi_1)\widetilde{u}_1,
\end{equation*}
where the right hand side is understood pointwise.
For each $a\geqslant \ell_n+1$ we integrate by parts bearing in mind the equation for $\phi_n$ in (\ref{3.1}),
\begin{equation}\label{v3.17}
\begin{aligned}
\op{T}g=&\int\limits_{\Pi_a} g\phi_n \di x=\e^2\l \int\limits_{\Pi_a} (\xi_1-1) \widetilde{u}_1\phi_n\di x
\\
&+\int\limits_0^1 \phi_n(0,x_2) \left(\frac{\p \widetilde{u}_1}{\p x_1}(+0,x_2,\e)-\frac{\p \widetilde{u}_1}{\p x_1}(-0,x_2,\e)\right)\di x_2
\\
&+\int\limits_0^{\ell_n} \phi_n(x_1,0) \frac{\p \widetilde{u}_1}{\p x_2}(x_1,0,\e)\di x_1
-\int\limits_{-\ell_n}^0 \phi_n(x_1,1) \frac{\p \widetilde{u}_1}{\p x_2}(x_1,1,\e)\di x_1.
\end{aligned}
\end{equation}
By direct calculations we check that
\begin{align*}
&\left(\frac{\p \widetilde{u}_1}{\p x_1}(+0,x_2,\e)-\frac{\p \widetilde{u}_1}{\p x_1}(-0,x_2,\e)\right)
\\
&\hphantom{\frac{\p (\widetilde{u}_1)}{\p x_1}(+0)}
=\sqrt{2}\sum\limits_{m=1}^{\infty}\bigg( \sin\sqrt{E_m}  x_2\int\limits_0^{+\infty} \E^{-k_m(\e\sqrt{-\l})|t_1|} f_m(\e t_1)\di t_1
\\
&\hphantom{\frac{\p (\widetilde{u}_1)}{\p x_1}(+0)=\sum\limits_{m=1}^{\infty}\bigg(}
+\sin\sqrt{E_m} (1-x_2)\,
\int\limits_{-\infty}^0 \E^{-k_m(\e\sqrt{-\l})|t_1|} f_m(\e t_1)\di t_1
\bigg)
\\
&\hphantom{\frac{\p (\widetilde{u}_1)}{\p x_1}(+0)}
=\sqrt{2}\sum\limits_{m=2}^{\infty}\bigg( \sin\sqrt{E_m}  x_2\int\limits_0^{+\infty} \E^{-k_m(\e\sqrt{-\l})|t_1|} f_m(\e t_1)\di t_1
\\
&\hphantom{\frac{\p (\widetilde{u}_1)}{\p x_1}(+0)=\sum\limits_{m=2}^{\infty}\bigg(}
+\sin\sqrt{E_m} (1-x_2)\,
\int\limits_{-\infty}^0 \E^{-k_m(\e\sqrt{-\l})|t_1|} f_m(\e t_1)\di t_1
\bigg)
\\
&\hphantom{\frac{\p (\widetilde{u}_1)}{\p x_1}(+0)=}
+2\sqrt{-2\l}\,\e^{-3/2} F_+(\e)\sin \frac{\pi}{2}x_2+ 2\sqrt{-2\l}\,\e^{-3/2} F_-(\e)\sin \frac{\pi}{2}(1-x_2),
\\
&\frac{\p \widetilde{u}_1}{\p x_2}(x_1,0,\e)=\sqrt{2}\sum\limits_{m=1}^{\infty} \sqrt{E_m} U_m(x_1,\e), \quad x_1>0,
\\
-&\frac{\p \widetilde{u}_1}{\p x_2}(x_1,1,\e)=\sqrt{2}\sum\limits_{m=1}^{\infty} \sqrt{E_m} U_m(x_1,\e), \quad x_1<0.
\end{align*}
We substitute these identities into (\ref{v3.17}),
\begin{align*}
\op{T}g=&\e^2\l \int\limits_{\Pi_a} (\xi_1-1) \widetilde{u}_1\phi_n\di x
\\
&+\sqrt{2}\sum\limits_{m=2}^\infty \int\limits_{0}^{1}
\phi_n(0,x_2)\sin\sqrt{E_m}  x_2\di x_2\,\int\limits_0^{+\infty} \E^{-k_m(\e\sqrt{-\l})|t_1|} f_m(\e t_1)\di t_1
\\
&+\sqrt{2}\sum\limits_{m=2}^\infty \int\limits_{0}^{1}
\phi_n(0,x_2)\sin\sqrt{E_m} (1-x_2)\di x_2\,
\int\limits_{-\infty}^0 \E^{-k_m(\e\sqrt{-\l})|t_1|} f_m(\e t_1)\di t_1
\\
&+\sqrt{2}\sum\limits_{m=2}^{\infty} \sqrt{E_m}
\int\limits_0^{\ell_n} \phi_n(x_1,0) U_m(x_1,\e) \di x_1
+\sqrt{2}\sum\limits_{m=2}^{\infty} \sqrt{E_m}
\int\limits_{-\ell_n}^0 \phi_n(x_1,1) U_m(x_1,\e) \di x_1
\\
&+2\sqrt{-2\l}\, \e^{-3/2}\left(F_+(\e)\int\limits_0^1 \phi_n(0,x_2)  \sin\frac{\pi}{2}x_2  \di x_2
+F_-(\e)\int\limits_0^1 \phi_n(0,x_2)  \sin\frac{\pi}{2}(1-x_2)  \di x_2\right)
\\
&+\frac{\pi}{\sqrt{2}}\int\limits_0^{\ell_n} \phi_n(x_1,0)  U_1(x_1,\e) \di x_1+\frac{\pi}{\sqrt{2}}\int\limits_{-\ell_n}^0 \phi_n(x_1,1) U_1(x_1,\e)\di x_1.
\end{align*}
We employ Lemma~\ref{lm3.1} to rewrite the last expression as follows,
\begin{equation}\label{v4.36}
\begin{aligned}
\op{T}g=&\e^{-3/2}\sqrt{-2\l}\big(F_+(\e)+(-1)^{n-1}F_-(\e)\big)+\e^2\l \int\limits_{\Pi_a} (\xi_1-1) \widetilde{u}_1\phi_n\di x
\\
&+\sum\limits_{m=2}^\infty \int\limits_{0}^{1}
\phi_n(0,x_2)\sin\sqrt{E_m}  x_2\di x_2\,\int\limits_0^{+\infty} \E^{-k_m(\e\sqrt{-\l})|t_1|} f_m(\e t_1)\di t_1
\\
&+\sum\limits_{m=2}^\infty \int\limits_{0}^{1}
\phi_n(0,x_2)\sin\sqrt{E_m} (1-x_2)\di x_2\,
\int\limits_{-\infty}^0 \E^{-k_m(\e\sqrt{-\l})|t_1|} f_m(\e t_1)\di t_1
\\
&+\sum\limits_{m=2}^{\infty} \sqrt{E_m}
\int\limits_0^{\ell_n} \phi_n(x_1,0) U_m(x_1,\e) \di x_1
\\
&+\sum\limits_{m=2}^{\infty} \sqrt{E_m}
\int\limits_{-\ell_n}^0 \phi_n(x_1,1) U_m(x_1,\e) \di x_1
\\
&+\frac{\pi}{2} \int\limits_0^{\ell_n} \phi_n(x_1,0)  \big(U_1(x_1,\e)-2\e^{-3/2}\sqrt{-\l}\,F_+(\e)x_1\big) \di x_1
\\
&+ \frac{\pi}{2}\int\limits_{-l}^0 \phi_n(x_1,1) \big(U_1(x_1,\e)+2\e^{-3/2}\sqrt{-\l}\,F_-(\e)x_1\big)\di x_1.
\end{aligned}
\end{equation}
Let us estimate each term in the last identity. Due to (\ref{v4.35}) and (\ref{v3.18}) we have
\begin{equation*}
\bigg|\e^2\l \int\limits_{\Pi_a} (\xi_1-1) \widetilde{u}_1\phi_n\di x\bigg| \leqslant C\e^{-1/2}\|f\|_{L_2(\Pi^{(\e)})}.
\end{equation*}
By Schwarz inequality we obtain
\begin{align*}
&\left|\int\limits_0^{+\infty} \E^{-k_m(\e\sqrt{-\l})|t_1|} f_m(\e t_1)\di t_1\right|^2\leqslant \frac{\|f_m(\e\,\cdot)\|^2_{L_2(\mathds{R}_+)}}{2\RE k_m(\e\sqrt{-\l})},\nonumber
\\
&\left|\int\limits_{-\infty}^0 \E^{-k_m(\e\sqrt{-\l})|t_1|} f_m(\e t_1)\di t_1\right|^2\leqslant \frac{\|f_m(\e\,\cdot)\|^2_{L_2(\mathds{R}_-)}}{2\RE k_m(\e\sqrt{-\l})}.\nonumber
\end{align*}
Hence, by Parseval identity (\ref{v3.25}),
\begin{align*}
&\left|\sum\limits_{m=2}^\infty \int\limits_{0}^{1}
\phi_n(0,x_2)\sin\sqrt{E_m}  x_2\di x_2\,\int\limits_0^{+\infty} \E^{-k_m(\e\sqrt{-\l})|t_1|} f_m(\e t_1)\di t_1
\right|\leqslant C\e^{-1}\|f\|_{L_2(\Pi^{(\e)})},
\\
&\left|\sum\limits_{m=2}^\infty \int\limits_{0}^{1}
\phi_n(0,x_2)\sin\sqrt{E_m} (1-x_2)\di x_2\,
\int\limits_{-\infty}^0 \E^{-k_m(\e\sqrt{-\l})|t_1|} f_m(\e t_1)\di t_1
\right|\leqslant C\e^{-1}\|f\|_{L_2(\Pi^{(\e)})}.
\end{align*}
The first estimate in (\ref{v4.1}) and Parseval identity (\ref{v3.25}) yield
\begin{align*}
\bigg|&\sum\limits_{m=2}^{\infty} \sqrt{E_m}
\int\limits_0^{\ell_n} \phi_n(x_1,0) U_m(x_1,\e) \di x_1
+\sum\limits_{m=1}^{\infty} \sqrt{E_m}
\int\limits_{-\ell_n}^0 \phi_n(x_1,1) U_m(x_1,\e) \di x_1\bigg|
\\
&\leqslant C\sum\limits_{m=2}^{\infty} \sqrt{E_m}\|U_m\|_{L_2(\mathds{R})}
\leqslant C\sum\limits_{m=2}^{\infty} E_m^{-1/2}\|f_m(\e\cdot)\|_{L_2(\mathds{R})}^2
&
\\
&\leqslant C \left(\sum\limits_{m=2}^{\infty}E_m^{-1}\right)^{1/2}
\left(\sum\limits_{m=2}^{\infty} \|f^{(m)}(\e\cdot)\|_{L_2(\mathds{R})}^2\right)^{1/2}
\leqslant  C\e^{-1}\|f\|_{L_2(\Pi^{(\e)})}.
\end{align*}
It remains to estimate two last terms in (\ref{v4.36}). It follows from the representation (\ref{v4.5}) that
\begin{align*}
U_1(x_1,\e)-&2\e^{-3/2}\sqrt{-\l} F_+(\e) x_1 =\E^{-\mu x_1} \int\limits_{0}^{x_1} \frac{\sinh \mu t_1}{\mu} f_1(\e t_1)\di t_1
\\
&- x_1 \int\limits_{0}^{x_1} \E^{-\mu t_1} f_1(\e t_1)\di t_1 + \left( \frac{\sinh\mu x_1}{\mu}-x_1
\right) \int\limits_{x_1}^{+\infty} \E^{-\mu t_1} f_1(\e t_1) \di t_1,\quad x_1>0,
\\
U_1(x_1,\e)+&2\e^{-3/2}\sqrt{-\l} F_-(\e) x_1=\E^{\mu x_1} \int\limits_{x_1}^{0} \frac{\sinh \mu t_1}{\mu} f_1(\e t_1)\di t_1
\\
&+ x_1 \int\limits_{0}^{x_1} \E^{\mu t_1} f_1(\e t_1)\di t_1 + \left( x_1-\frac{\sinh\mu x_1}{\mu}
\right) \int\limits_{-\infty}^{x_1} \E^{\mu t_1} f_1(\e t_1) \di t_1,\quad x_1<0,
\end{align*}
where $\mu=\e\sqrt{-\l}$. These formulas and Parseval identity (\ref{v3.25}) imply the desired estimate for the two last terms in (\ref{v4.36}),
\begin{align*}
\bigg|&\frac{\pi}{\sqrt{2}} \int\limits_0^{\ell_n} \phi_n(x_1,0)  \big(U_1(x_1,\e)-2\e^{-3/2}\sqrt{-\l}F_+(\e)\big) \di x_1
\\
&\hphantom{\bigg|}+ \frac{\pi}{\sqrt{2}}\int\limits_{-\ell_n}^0 \phi_n(x_1,1) \big(U_1(x_1,\e)+2\e^{-3/2}\sqrt{-\l} x_1 F_-(\e)\big)\di x_1\bigg|
\\
&\leqslant C\bigg(\|U_1-2\e^{-3/2}\sqrt{-\l} F_+(\e) x_1\|_{L_2(0,\ell_n)}
\\
&\hphantom{\leqslant C\bigg(}+ \|U_1(x_1,\e)+2\e^{-3/2}\sqrt{-\l} F_-(\e) x_1\|_{L_2(-\ell_n,0)}\bigg)
\leqslant C\e^{-1}\|f\|_{L_2(\Pi^{(\e)})}.
\end{align*}
The proof is complete.
\end{proof}

We again employ the results of \cite{pap1}. Namely, it follows from the identities (4.5), (4.9) and Lemma~4.7 in this paper and the inequality (\ref{v4.31}) that the function $\widetilde{u}_2$ can be represented as
\begin{equation}\label{4.5}
\widetilde{u}_2(x,\e)=\sum\limits_{m=1}^{\infty} \widetilde{c}_m^\pm(\e,\l)\chi_m(x) \E^{-k_m|x_1|},\quad \pm x_1>\ell,
\end{equation}
where $\widetilde{c}_m^\pm(\e,\l)$ are certain coefficients satisfying the estimates
\begin{equation}\label{4.6}
\begin{aligned}
&\left|\widetilde{c}_1^+(\e,\l)- \frac{\op{T}g}{\e\sqrt{-\l}}
\right|\leqslant C\e^{-3/2}\|f\|_{L_2(\Pi^{(\e)})},
\\
&\left|\widetilde{c}_1^-(\e,\l)- (-1)^{n-1}\frac{\op{T}g}{\e\sqrt{-\l}}
\right|\leqslant C\e^{-3/2}\|f\|_{L_2(\Pi^{(\e)})},
\\
&\sum\limits_{m=2}^{\infty} m \left(|\widetilde{c}_m^+(\e,\l)|^2+ |\widetilde{c}_m^-(\e,\l)|^2\right) \leqslant C\e^{-5/2} \|f\|_{L_2(\Pi^{(\e)})}
\end{aligned}
\end{equation}
for critical $\ell=\ell_n$, and
\begin{equation}\label{4.7}
\sum\limits_{m=1}^{\infty} m \left(|\widetilde{c}_m^+(\e,\l)|^2+ |\widetilde{c}_m^-(\e,\l)|^2\right) \leqslant C\e^{-3/2} \|f\|_{L_2(\Pi^{(\e)})}
\end{equation}
for noncritical $\ell$. Thus, denoting
\begin{equation*}
\widetilde{u}_2^\bot(x,\e):=\sum\limits_{m=1}^{\infty} \widetilde{c}_m^\pm(\e)\chi_m(x) \E^{-\sqrt{E_m-E_1-\e^2\l}|x_1|},\quad \pm x_1\geqslant \ell,
\end{equation*}
we have the inequalities
\begin{equation}\label{4.8}
\begin{aligned}
&\|\widetilde{u}_2^\bot\|_{\H^1(\Pi\setminus\Pi_\ell)} \leqslant C\e^{-5/2} \|f\|_{L_2(\Pi^{(\e)})}\quad \text{for critical}\quad \ell,
\\
&\|\widetilde{u}_2^\bot\|_{\H^1(\Pi\setminus\Pi_\ell)} \leqslant C\e^{-3/2} \|f\|_{L_2(\Pi^{(\e)})}\quad \text{for noncritical}\quad \ell.
\end{aligned}
\end{equation}
It follows from (\ref{v3.15}), (\ref{v3.23}), (\ref{v4.35}) and Lemma~\ref{vlm4.4} that for critical $\ell$
\begin{equation}\label{4.9}
|\op{T}g|\leqslant C\e^{-3/2}\|f\|_{L_2(\Pi^{(\e)})},\quad \|\widetilde{u}_2\|_{\H^1(\Pi_a)} \leqslant C \e^{-5/2} \|f\|_{L_2(\Pi^{(\e)})},
\end{equation}
where $a$ is arbitrary but fixed. Denote
\begin{equation}\label{4.14}
\widehat{U}_\e(x_1):=U_1(x,\e)
\end{equation}
for noncritical $\ell$, and for critical $\ell=\ell_n$ we let
\begin{equation}\label{4.15}
\begin{gathered}
\op{T}_0 f:=\sqrt{2} \big(F_+(\e)+(-1)^{n-1}F_-(\e)\big),
\\
\widehat{U}_\e(x_1):=\left\{
\begin{aligned}
& U_1(x_1,\e)+\e^{-5/2} \E^{-\e\sqrt{-\l}x_1}\op{T}_0f, && x_1>0,
\\
& U_1(x_1,\e)+(-1)^{n-1}\e^{-5/2} \E^{\e\sqrt{-\l}x_1}\op{T}_0f, && x_1<0.
\end{aligned}
\right.
\end{gathered}
\end{equation}
It follows from Lemmas~\ref{vlm3.8},~\ref{vlm4.4} and the relations (\ref{v3.5a}), (\ref{v3.15}), (\ref{v3.23}), (\ref{4.2}), (\ref{4.7}), (\ref{4.9}) that
\begin{equation}\label{4.10}
\begin{aligned}
&\|(1-\xi_1) \widetilde{u}_1\|_{\H^1(\Pi)} \leqslant C\e^{-3/2} \|f\|_{L_2(\Pi^{(\e)})},
\\
&\|\widetilde{u}_2\|_{\H^1(\Pi_{2\ell})} \leqslant C\e^{-5/2} \|f\|_{L_2(\Pi^{(\e)})},
\\
&\|\widehat{U}_\e\chi_1\|_{\H^1(\Pi_{2\ell}\setminus Ox_1)}  \leqslant C\e^{-5/2}\|f\|_{L_2(\Pi^{(\e)})},
\end{aligned}
\end{equation}
if $\ell$ is critical, and
\begin{equation}\label{4.4}
\begin{aligned}
&\|(1-\xi_1) \widetilde{u}_1\|_{\H^1(\Pi)} \leqslant C\e^{-3/2} \|f\|_{L_2(\Pi^{(\e)})},
\\
&\|\widetilde{u}_2\|_{\H^1(\Pi_{2\ell})} \leqslant C\e^{-3/2} \|f\|_{L_2(\Pi^{(\e)})},
\\
&\|\widehat{U}_\e\chi_1\|_{\H^1(\Pi_{2\ell})}  \leqslant C\e^{-3/2}\|f\|_{L_2(\Pi^{(\e)})},
\end{aligned}
\end{equation}
if $\ell$ is noncritical. Together with (\ref{4.5}), (\ref{4.6}), (\ref{4.7}), (\ref{4.8}) it yields
\begin{equation}\label{4.42}
\|\widetilde{u}_\e-\widehat{U}_\e\chi_1\|_{\H^1(\Pi\setminus Ox_1)} \leqslant C\e^{-5/2}\|f\|_{L_2(\Pi^{(\e)})}
\end{equation}
for critical $\ell$, and
\begin{equation}\label{4.43}
\|\widetilde{u}_\e-\widehat{U}_\e\chi_1\|_{\H^1(\Pi)} \leqslant C\e^{-3/2}\|f\|_{L_2(\Pi^{(\e)})}
\end{equation}
for noncritical $\ell$. We return back to the original variables $x$ in $\Pi^{(\e)}$, let
\begin{equation}\label{4.3}
U_\e(x):=\e^{5/2}\widehat{U}_\e(x_1\e^{-1}),
\end{equation}
and employ (\ref{4.1}) to rewrite the last estimates as
follows,
\begin{equation}\label{4.12}
\|u_\e-U_\e\chi_1^{\e}\|_{L_2(\Pi^{(\e)})}\leqslant C\e^{1/2}\|f\|_{L_2(\Pi^{(\e)})}
\end{equation}
for critical $\ell$, and
\begin{equation}\label{4.13}
\begin{aligned}
&\|u_\e-U_\e\chi_1^{(\e)}\|_{L_2(\Pi^{(\e)})} \leqslant C\e^{3/2}\|f\|_{L_2(\Pi^{(\e)})},
\\
&\|u_\e-U_\e\chi_1^{(\e)}\|_{L_2(\Pi^{(\e)})} \leqslant C\e^{1/2}\|f\|_{L_2(\Pi^{(\e)})}
\end{aligned}
\end{equation}
for noncritical $\ell$.

Due to (\ref{4.3}), (\ref{4.14}), and (\ref{v4.43}) for noncritical $\ell$ the function $U_\e$ solves the boundary value problem
\begin{equation}\label{4.16}
\left(-\frac{d^2}{dx^2}-\l\right)U_\e=\e^{-1/2}f_1^\pm=(\op{P}_\e f),\quad \pm x_1>0,\qquad U_\e(0)=0.
\end{equation}
Thus,
\begin{equation}\label{4.17}
U_\e=(\efop-\l)^{-1} \op{P}_\e f
\end{equation}
for noncritical $\ell$. For critical $\ell$ the function $U_\e$ solves the same equation as in (\ref{4.16}), while by (\ref{v4.43}), (\ref{v4.34}), (\ref{4.15}) the boundary conditions are
\begin{align*}
&U_\e(+0)=(-1)^{n-1}U_\e(-0),
\quad
U_\e'(+0)=\sqrt{-\l}\big(2F_+(\e)-\op{T}_0f\big),
\\
&U_\e'(-0)=\sqrt{-\l}\big(-2F_-(\e)+(-1)^{n-1}\op{T}_0f\big)=(-1)^{n-1}
U_\e'(+0).
\end{align*}
Hence, for critical $\ell=\ell_n$ we again arrive at (\ref{4.17}). Together with the estimates (\ref{4.12}), (\ref{4.13}) it completes the proof of Theorem~\ref{th2.1}.

\section{Proof of Theorem~\ref{th2.2}}

The main lines of the proof are the same as in the case of Theorem~\ref{th2.1}. We first rescale the variables $x\mapsto x\e^{-1}$ and obtain
\begin{equation*}
u_\e=\e^2 \widetilde{u}_\e(\e\cdot),\quad \widetilde{u}_\e:= \left(\op{H}_{\e^{-1}L}- E+\e^2\l\right)^{-1} f(\e\cdot).
\end{equation*}
The function $u_1$ is introduced as the solution to the boundary value problem
\begin{equation}\label{5.2}
\begin{gathered}
\left(-\D-E+\e^2\l\right)\widetilde{u}_1=f\quad\text{in}\quad \Pi,
\\
\widetilde{u}_1=0\quad\text{on}\quad\g_{L\e^{-1}}\cup\{x: |x_1|=L\e^{-1},\ 0<x_2<1\},\qquad
\frac{\p \widetilde{u}_1}{\p x_2}=0\quad\text{on}\quad \G_{L\e^{-1}}.
\end{gathered}
\end{equation}
The solution is given explicitly by the separation of variables
\begin{align}
&\widetilde{u}_1(x,\e)=\sum\limits_{m=1}^{\infty} U_m(x,\e)\chi_{m,L}(x,\e),\label{5.3}
\\
&
U_m(x_1,\e)=\left\{
\begin{aligned}
&\int\limits_0^{+\infty} G_m^+(x_1,t_1,\e) f_{m,L}(\e t_1)\di t_1, &&
 x_1>L\e^{-1},
\\
&\int\limits_{-L\e^{-1}}^{L\e^{-1}} G_m^0(x_1,t_1,\e) f_{m,L}(\e t_1)\di t_1, && |x_1|<L\e^{-1},
\\
&  \int\limits_{-\infty}^0G_m^-(x_1,t_1,\e,) f_{m,L}(\e t_1)\di t_1, && x_1<-L\e^{-1},
\end{aligned}
\right.\label{5.4}
\\
&
G_m^+(x_1,t_1,\e):= \frac{\E^{-k_{m,L}|x_1-L\e^{-1}-t_1|}-\E^{-k_{m,L} (x_1-L\e^{-1}+t_1)}}{2k_{m,L}},\nonumber
\\
&G_m^0(x_1,t_1,\e):= - \frac{ \E^{-K_m (L\e^{-1}-t_1)} \sinh K_m(x_1+L \e^{-1})}{2K_m\sinh 2 K_m L\e^{-1}}\nonumber
\\
&\hphantom{G_m^0(x_1,t_1,\e):=} - \frac{ \E^{-K_m (L\e^{-1}+t_1)} \sinh K_m(x_1-L \e^{-1})}{2K_m\sinh 2 K_m L\e^{-1}} + \frac{\E^{-K_m|x_1-t_1|}}{2K_m}\nonumber
\\
&G_m^+(x_1,t_1,\e)=  \frac{\E^{-k_{m,L}|x_1+L\e^{-1}-t_1|} -\E^{k_{m,L}(x_1+L\e^{-1}+t_1)}}{2 k_{m,L}},\nonumber
\\
&
\chi_{m,L}(x):=\chi_m(x), \quad |x_1|>L\e^{-1},\qquad
\chi_{m,L}(x):=1, \quad |x_1|<L\e^{-1},\quad m=1,\nonumber
\\
&
\chi_{m,L}(x):=\sqrt{2}\cos\pi(m-1)x_2, \quad |x_1|<L\e^{-1},\quad m\geqslant 2,\nonumber
\\
&f_{m,L}(\e x_1):=\int\limits_0^1 \chi_{m,L}(x) f(\e x)\di x_2, \quad k_{m,L}:=\sqrt{E_m-E+\e^2\l},\nonumber
\\
&K_m:=\sqrt{\pi^2 (m-1)^2-E+\e^2\l},\quad m\geqslant 1, \nonumber
\end{align}
This series converges in the spaces $\H^2(\Pi^\pm_{\pm L\e^{-1}})$, $\H^2(\Pi_{L\e^{-1}})$, and $\H^1(\Pi)$, where $\Pi^\pm_a:=\Pi\cap\{x: \pm x_1>\pm a\}$, and we remind that $\Pi_a:=\Pi\cap\{x: |x_1|<a\}$.
 The Parseval identity
\begin{equation*}
\|f\|_{L_2(\Pi^{\e})}^2=\e^2\sum\limits_{m=1}^{\infty} \|f_{m,L}(\e\cdot)\|_{L_2(\mathds{R})}^2
\end{equation*}
holds true.
In what follows by $C$ we indicate various inessential constants independent of $\e$ and $f$.

The analogues of Lemma~\ref{vlm3.8} in this case are the next three statements.
\begin{lemma}\label{lm5.1}
Let $E=0$. Then the estimates
\begin{align}
&\|\widetilde{u}_1\|_{\H^1(\Pi^+_{L\e^{-1}})}+ \|\widetilde{u}_1\|_{\H^1(\Pi^-_{-L\e^{-1}})}\leqslant C\e^{-1} \|f\|_{L_2(\Pi^{(\e)})},\label{5.6}
\\
&\|\widetilde{u}_1-U_1\chi_{1,L}\|_{\H^1(\Pi_{L\e^{-1}})} \leqslant C\e^{-1}\|f\|_{L_2(\Pi^{(\e)})},\label{5.7}
\end{align}
hold true.
\end{lemma}

\begin{proof}
We first prove the estimate (\ref{5.6}). It follows from (\ref{5.2}), (\ref{5.3}), (\ref{5.4}) that
\begin{equation}\label{5.8}
\|\nabla \widetilde{u}_1\|_{\H^1(\Pi^\pm_{\pm L\e^{-1}})}^2 - \e^2\l \| \widetilde{u}_1\|_{L_2(\Pi^\pm_{\pm L\e^{-1}})}^2 = \big( f(\e\cdot), \widetilde{u}_1 \big)_{L_2\big(\Pi^\pm_{\pm L\e^{-1}}\big)}.
\end{equation}
Due to the boundary condition in (\ref{5.2}) and the minimax principle we have
\begin{equation*}
\left\|\frac{\p \widetilde{u}_1}{\p x_2}(x_1,\cdot,\e)\right\|_{L_2(0,1)}^2\geqslant \frac{\pi^2}{4} \|\widetilde{u}_1(x_1,\cdot,\e)\|_{L_2(0,1)}^2\quad \text{for}\quad |x_1|>L\e^{-1}.
\end{equation*}
Together with (\ref{5.8}) it implies
\begin{align*}
\left(\frac{\pi^2}{4}-\e^2\l\right) \|\widetilde{u}_1\|_{L_2\big(\Pi^\pm_{\pm L\e^{-1}}\big)}^2 &\leqslant \Big| \big(f(\e\cdot), \widetilde{u}_1)_{L_2\left(\Pi^\pm_{\pm L\e^{-1}}\right)}\Big|
\\
&\leqslant \|f(\e\cdot)\|_{L_2\big(\Pi^\pm_{\pm L\e^{-1}}\big)} \|\widetilde{u}_1\|_{L_2\big(\Pi^\pm_{\pm L\e^{-1}}\big)}.
\end{align*}
Hence,
\begin{equation*}
\|\widetilde{u}_1\|_{L_2\left(\Pi^\pm_{\pm L\e^{-1}}\right)} \leqslant C\|f(\e\cdot)\|_{L_2\left(\Pi^\pm_{\pm L\e^{-1}}\right)}\leqslant C\e^{-1}\|f\|_{L_2(\Pi^{(\e)})}.
\end{equation*}
Substituting this estimate into (\ref{5.8}), we arrive at (\ref{5.6}).

The proof of the estimate (\ref{5.7}) is analogous to that of the estimate (\ref{v3.24a}) in Lemma~\ref{vlm3.8}.\end{proof}

\begin{lemma}\label{lm5.3}
Let $E=\pi^2/4$. Then the estimates
\begin{align*}
&\|\widetilde{u}_1-U_1\chi_{1,L}\|_{\H^1(\Pi^+_{L\e^{-1}})}+ \|\widetilde{u}_1-U_1\chi_{1,L}\|_{\H^1(\Pi^-_{-L\e^{-1}})}\leqslant C\e^{-1} \|f\|_{L_2(\Pi^{(\e)})},
\\
&\|\widetilde{u}_1\|_{\H^1(\Pi_{L\e^{-1}})} \leqslant C\e^{-1}\|f\|_{L_2(\Pi^{(\e)})},
\end{align*}
hold true.
\end{lemma}

The proof of this lemma is analogous to that of Lemma~\ref{lm5.1}.

\begin{lemma}\label{lm5.2}
Let $E=0$ or $E=\pi^2/4$. Then the estimate
\begin{equation*}
\|U_1\chi_{1,L}\|_{\H^1(\Pi_{L\e^{-1}+2}\setminus\Pi_{L\e^{-1}-2})} \leqslant C\e^{-1}\|f\|_{L_2(\Pi^{(\e)})}
\end{equation*}
holds true.
\end{lemma}

\begin{proof}
It follows from (\ref{5.4}) that for $x_1=L\e^{-1}+z$, $|z|<2$, the function $U_1$ can be represented as
\begin{equation*}
U_1(L\e^{-1}+z,\e)=\int\limits_{0}^{+\infty} \frac{\E^{-k_{1,L}|z-t_1|}-\E^{-k_{1,L}(z+t_1)}}{2 k_{1,L}} f_{1,L}(\e t_1)\di t_1,
\end{equation*}
while for $x_1=-L\e^{-1}+z$, $|z|<2$, it satisfies the identity
\begin{equation*}
U_1(-L\e^{-1}+z,\e)=\int\limits_{-\infty}^{0} \frac{\E^{-k_{1,L}|z-t_1|}-\E^{k_{1,L}(z+t_1)}}{2 k_{1,L}} f_{1,L}(\e t_1)\di t_1.
\end{equation*}
These formulas imply the desired estimate
\begin{equation*}
\|U_1\chi_{1,L}\|_{\H^1(\Pi_{L\e^{-1}+2}\setminus\Pi_{L\e^{-1}-2})} \leqslant C\|f_{1,L}(\e\cdot)\|_{L_2(\mathds{R})} \leqslant C\e^{-1}\|f\|_{L_2(\Pi^{(\e)})}.
\end{equation*}
\end{proof}

Let $\xi_2=\xi_2(x_1)$ be an infinitely differentiable cut-off function being one as $|x_1|>2$ and vanishing as $|x_1|<1$. We construct the function $\widetilde{u}_\e$ as
\begin{equation}\label{5.11}
\widetilde{u}_\e(x)=\widetilde{u}_1(x)\xi_3(x_1,\e)+\widetilde{u}_2(x), \quad \xi_3(x_1,\e):=\xi_2(x_1-L\e^{-1})+\xi_2(x_1+L\e^{-1}).
\end{equation}
Then the function $\widetilde{u}_2$ is given by the identity
\begin{equation}\label{5.22}
\widetilde{u}_2(x)=(\op{H}_{L\e^{-1}}-E-\e^2\l)^{-1}g,\quad g:=f(\e\cdot)-(-\D-\e^2\l)^{-1}\widetilde{u}_1\xi_3.
\end{equation}
Lemma~\ref{lm5.2} and the problem (\ref{5.2}) allow us to estimate $g$,
\begin{gather}
g=(1-\xi_3)f+\left(2\xi_3'\frac{\p}{\p x_3}+\xi_3''\right)\widetilde{u}_1,\label{5.12}
\\
\|g\|_{L_2(\Pi)}\leqslant C\left(\|\widetilde{u}_1\|_{\H^1(\Pi_{L\e^{-1}+2}\setminus\Pi_{L\e^{-1}-2})}
+\|f(\e\cdot)\|_{L_2(\Pi)}\right) \leqslant C\e^{-1} \|f\|_{L_2(\Pi^{(\e)})}.\label{5.13}
\end{gather}
We observe that due to the definition of $\xi_3$ and (\ref{5.12}) the function $g$ is supported in $\Pi_{L\e^{-1}+2}\setminus\Pi_{L\e^{-1}-2}$. We then consider the function $g$ separately for $x_1>0$ and $x_1<0$. Namely, we let
\begin{equation}\label{5.14}
g_+(x):=g(x_1-L\e^{-1}, x_2),\quad g_-(x):=g(L\e^{-1}-x_1,1-x_2),\quad |x_1|<2.
\end{equation}
Both these functions are extended by zero for $|x_1|>2$. It is clear that
\begin{equation*}
g(x)=g_+(x_1+L\e^{-1},x_2)+g_-(L\e^{-1}-x_1,1-x_2).
\end{equation*}
Let $\widetilde{u}_3^\pm$ be the solutions to the problem (\ref{3.5}), (\ref{3.6}) for $E=0$ and to the problem
(\ref{3.5}), (\ref{3.9}) for $E=\pi^2/4$ with $\mu=\e\sqrt{-\l}$ and $h=g_\pm$. By Lemmas~\ref{lm3.2},~\ref{lm3.3} and by (\ref{5.13}) we conclude that for 
each $b>0$ the estimate
\begin{equation}\label{5.17}
\|\widetilde{u}_3^\pm\|_{\H^1(\Pi_b)
}\leqslant C\e^{-1}\|f\|_{L_2(\Pi^{(\e)})}
\end{equation}
holds true. Moreover, subject to the value of $E$ these functions satisfy either the representation (\ref{3.7}) or (\ref{3.14}) with the coefficients satisfying (\ref{3.8}). The last estimate and (\ref{5.13}) imply
\begin{equation}\label{5.18}
|c_0^-(\e\sqrt{-\l})|^2+ \sum\limits_{m=1}^{\infty} m \big(|c_m^+(\e\sqrt{-\l})|^2+|c_m^-(\mu)|^2\big) \leqslant C
\e^{-2}\|f\|_{L_2(\Pi^{(\e)})}^2.
\end{equation}

By $\xi_4=\xi_4(t)$ we denote an infinitely differentiable cut-off function being one as $t>1$ and vanishing for $t<0$.
We introduce one more function
\begin{equation}\label{5.19}
\begin{aligned}
\widetilde{u}_3(x,\e):=&\xi_4\left(\frac{x_1+L\e^{-1}-3}{2L\e^{-1}-6}
\right)\widetilde{u}_3^+(x_1-L\e^{-1},x_2)
\\
&+ \xi_4\left(\frac{3-L\e^{-1}-x_1}{2L\e^{-1}-6}
\right)\widetilde{u}_3^+(L\e^{-1}-x_1,1-x_2).
\end{aligned}
\end{equation}
It follows from the definition of $\widetilde{u}_3^\pm$ and $\xi_4$ that the function $\widetilde{u}_3$ belongs to the domain of $\op{H}_{L\e^{-1}}$ and
\begin{gather}
(\op{H}_{L\e^{-1}}-E-\e^2\l)\widetilde{u}_3=g+\widehat{g},\label{5.20}
\\
\begin{aligned}
\widehat{g}(x):=&\frac{2}{2L\e^{-1}-6} \xi_4' \left(\frac{x_1+L\e^{-1}-3}{2L\e^{-1}-6}
\right)\frac{\p\widetilde{u}_3^+}{\p x_1}(x_1-L\e^{-1},x_2)
\\
&+
\frac{1}{(2L\e^{-1}-6)^2} \xi_4'' \left(\frac{x_1+L\e^{-1}-3}{2L\e^{-1}-6}
\right)\widetilde{u}_3^+(x_1-L\e^{-1},x_2)
\\
&+\frac{2}{2L\e^{-1}-6} \xi_4' \left(\frac{3-L\e^{-1}-x_1}{2L\e^{-1}-6}
\right)\frac{\p\widetilde{u}_3^+}{\p x_1}(L\e^{-1}-x_1,1-x_2)
\\
&+
\frac{1}{(2L\e^{-1}-6)^2} \xi_4'' \left(\frac{3-L\e^{-1}-x_1}{2L\e^{-1}-6}
\right)\widetilde{u}_3^+(L\e^{-1}-x_1,1-x_2).
\end{aligned}\nonumber
\end{gather}
The function $\widehat{g}$ is supported in $\Pi_{L\e^{-1}+2}\setminus\Pi_{L\e^{-1}-2}$. Bearing this fact and (\ref{5.17}), (\ref{5.18}) in mind together with the representations (\ref{3.7}) and (\ref{3.14}), we can estimate $\widehat{g}$ as
\begin{equation*}
\|\widehat{g}\|_{L_2(\Pi)}\leqslant C\e^{-1}\|f\|_{L_2(\Pi^{(\e)})}.
\end{equation*}
The last estimate and the assertions (\ref{5.22}), (\ref{5.20}) yield
\begin{align*}
&\widetilde{u}_3-\widetilde{u}_2=(\op{H}_{L\e^{-1}}-E-\e^2\l)^{-1}\widehat{g},
\\
&\|\widetilde{u}_3-\widetilde{u}_2\|_{L_2(\Pi)}\leqslant C\e^{-1}\|f\|_{L_2(\Pi^{(\e)})}.
\end{align*}
Since
\begin{equation*}
\|\nabla (\widetilde{u}_3-\widetilde{u}_2)\|_{L_2(\Pi)}^2 =(E-\e^2\l)\|(\widetilde{u}_3-\widetilde{u}_2)\|_{L_2(\Pi)}^2 + (\widehat{g}, \widetilde{u}_3-\widetilde{u}_2)_{L_2(\Pi)},
\end{equation*}
we get
\begin{equation}\label{5.23}
\|\widetilde{u}_3-\widetilde{u}_2\|_{\H^1(\Pi)}\leqslant C\e^{-1}\|f\|_{L_2(\Pi^{(\e)})}.
\end{equation}

Now we estimate $\widetilde{u}_3$. As $E=0$, it follows from Lemma~\ref{lm3.2} and (\ref{5.13}), (\ref{5.14}), (\ref{5.17}) that
\begin{align*}
& \|\widetilde{u}_\e^\pm\|_{\H^1(\Pi_0^+)} \leqslant C\|g_\pm\|_{L_2(\Pi)} \leqslant C\e^{-1}\|f\|_{L_2(\Pi^{(\e)})},
\\
& \|\widetilde{u}_\e^\pm\|_{\H^1(\Pi_0^+\cap\Pi_{L\e^{-1}+2})}  \leqslant C\big(\|g_\pm\|_{L_2(\Pi)} + \e^{-1/2} |c_0^-(\e\sqrt{-\l})|\big) \leqslant C\e^{-3/2}\|f\|_{L_2(\Pi^{(\e)})}.
\end{align*}
Hence, in view of the definition (\ref{5.19}) of $\widetilde{u}_3$ it satisfies the estimate
\begin{equation}\label{5.25}
\|\widetilde{u}_\e\|_{\H^1(\Pi)}\leqslant C\e^{-3/2}\|f\|_{L_2(\Pi^{(\e)})}.
\end{equation}
This estimate is also valid for $E=\pi^2/4$ that can be proven in the same way.

The inequalities (\ref{5.23}), (\ref{5.25}) imply
\begin{equation*}
\|\widetilde{u}_2\|_{\H^1(\Pi)}\leqslant C\e^{-3/2}\|f\|_{L_2(\Pi^{(\e)})}
\end{equation*}
and by (\ref{5.11}) it yields
\begin{equation*}
\|\widetilde{u}_\e-\widetilde{u}_1\xi_3\|_{\H^1(\Pi)} \leqslant C\e^{-3/2} \|f\|_{L_2(\Pi^{(\e)})}.
\end{equation*}
We apply Lemmas~\ref{lm5.1},~\ref{lm5.3},~\ref{lm5.2} and proceed in the same way as in (\ref{4.10}), (\ref{4.4}), (\ref{4.42}), (\ref{4.43}), (\ref{4.3}), (\ref{4.12}), (\ref{4.13}). It leads us to the estimates
\begin{equation*}
\|\widetilde{u}_\e\|_{\H^1(\Pi\setminus\Pi_{L\e^{-1}})} +  \|\widetilde{u}_\e-U_1\chi_{1,L}\|_{\H^1(\Pi_{L\e^{-1}})} \leqslant C\e^{-3/2} \|f\|_{L_2(\Pi^{(\e)})}
\end{equation*}
as $E=0$, and
\begin{equation*}
\|\widetilde{u}_\e-U_1\chi_{1,L}\|_{\H^1(\Pi\setminus\Pi_{L\e^{-1}})} + \|\widetilde{u}_\e\|_{\H^1(\Pi_{L\e^{-1}})} \leqslant C\e^{-3/2} \|f\|_{L_2(\Pi^{(\e)})}
\end{equation*}
as $E=\pi^2/4$. The desired inequalities follow directly from the obtained ones. The proof is complete.

\section*{Acknowledgments}

This work was initiated by stimulating discussions with David Krej\v{c}i\v{r}\'{\i}k. The authors thank him for this. They also thank the referee for valuable remarks.

The research was partially supported by the grant ``Spectral theory and asymptotic analysis'', FRA 2010 of Department of Engineering of the University of Sannio.

D.B. was partially supported by RFBR, the grant of the President of Russia for young scientists -- doctors of science and for leading scientific schools, by the Federal Task Program ``Scientific and pedagogical staff of innovative Russia for 2009-2013'' (contract no. 02.740.11.0612), and by the grant of FCT (ptdc/mat/101007/2008)

\end{document}